\documentclass[runningheads,12pt,a4paper]{article}
\usepackage{graphicx}
\usepackage{amsthm}
\usepackage{amsmath,amssymb}
\usepackage{enumitem}
\usepackage{xcolor}

\newtheorem{theorem}{Theorem}

\newtheorem{lemma}{Lemma}
\newtheorem{conjecture}{Conjecture}

\begin{document}
	\title{The Linear Arboricity Conjecture for Graphs with Large Girth}
	%
	%
	\author{Tapas Kumar Mishra \footnote{Orcid: 0000-0002-9825-3828} \\
		%
		%
		Department of Computer Science and Engineering,\\
		National Institute of Technology, Rourkela
		769008, India\\
		mishrat@nitrkl.ac.in\\
		https://mishra-tapas.github.io/}
	\maketitle              
	\begin{abstract}
		The Linear Arboricity Conjecture asserts that the linear arboricity of a graph with maximum degree $\Delta$ is $\lceil (\Delta+1)/2 \rceil$. For a $2k$-regular graph $G$, this implies $la(G) = k+1$. In this note, we utilize a network flow construction to establish upper bounds on $la(G)$ conditioned on the girth $g(G)$. We prove that if $g(G) \ge 2k$, the conjecture holds true, i.e., $la(G) \le k+1$. Furthermore, we demonstrate that for graphs with girth $g(G)$ at least $k$, $k/2$, $k/4$ and $2k/c$ for any integer constant $c$, the linear arboricity $la(G)$ satisfies the upper bounds $k+2$, $k+3$, $k+5$ and $k+\left\lceil \frac{3c+2}{2}\right\rceil$, respectively. Our approach relies on decomposing the graph into $k$ edge-disjoint 2-factors and constructing an auxiliary flow network with lower bound constraints to identify a sparse transversal subgraph that intersects every cycle in the decomposition.
		
	\end{abstract}
	
	\section{Introduction}
	
	A linear forest is a graph in which every connected component is a path. The linear arboricity of a graph $G$, denoted by $la(G)$, is defined as the minimum number of linear forests whose edge sets partition the edge set $E(G)$. This concept was introduced by Harary \cite{harary1970} as a natural refinement of the arboricity of a graph, covering $G$ with paths rather than general forests. For any graph $G$ with maximum degree $\Delta(G)$, simple counting arguments provide a trivial lower bound. Since every vertex $v$ has degree $d(v)$ and the maximum degree in a linear forest is 2 (at internal nodes) or 1 (at endpoints), at least $\lceil \Delta(G)/2 \rceil$ linear forests are required to cover the edges incident to a vertex of maximum degree. Thus, $la(G) \ge \lceil \Delta(G)/2 \rceil$.
	
	In 1981, Akiyama, Exoo, and Harary proposed the following conjecture, asserting that this lower bound is nearly tight:
	\begin{conjecture}[Linear Arboricity Conjecture \cite{akiyama1981}]
For any graph $G$ with maximum degree $\Delta$,$$la(G) = \left\lceil \frac{\Delta + 1}{2} \right\rceil.$$
	\end{conjecture}

	The conjecture has been verified for specific classes of graphs, including trees, complete graphs, and complete bipartite graphs \cite{akiyama1980,akiyama1981}. In 1988, Alon \cite{alon1988} proved using probabilistic methods that the conjecture holds asymptotically; that is, $la(G) = \frac{\Delta}{2} + o(\Delta)$ as $\Delta \to \infty$.
	Ferber, Fox and Jain \cite{ferber2020} established  an upper bound of $\Delta/2+O(\Delta^{.661})$, which was further improved to $\Delta/2+3\sqrt{\Delta}\log^4\Delta$
	by Lang and Postle \cite{lang2023} when $\Delta$ is sufficiently large.
	 However, the exact conjecture remains open for general graphs.
	
	This note focuses on the case of $2k$-regular graphs. For a $2k$-regular graph $G$, the lower bound is $\lceil 2k/2 \rceil = k$. However, since a union of $k$ linear forests has maximum degree $2k$ only if every vertex is an internal node of a path in every forest (implying the forests are spanning cycles), it is impossible to cover a regular graph with $k$ linear forests. Thus, for regular graphs, the conjecture specifically predicts:$$la(G) = k + 1.$$
	While probabilistic approaches like the Lov\'{a}sz Local Lemma are powerful for general bounds, constructive arguments often provide sharper insights for structured subclasses. 
	
	\subsection{Our results}
	In this paper, we explore the relationship between the girth of a graph, denoted $g(G)$, and its linear arboricity. We utilize Petersen's 2-Factor Theorem \cite{petersen1891} to decompose the $2k$ regular graph into $k$ edge-disjoint 2-factors, reducing the problem to finding a sparse transversal subgraph that intersects every cycle. We formulate this cycle-breaking problem as a network flow problem with lower bound constraints. By applying the Integrality Theorem for circulations \cite{hoffman1960}, we derive deterministic upper bounds for regular graphs with high girth. Our main results establish that if $g(G) \ge 2k$, the Linear Arboricity Conjecture holds (i.e., $la(G) \le k+1$). This is an improvement to the Theorem 2.1 of \cite{alon1988}. Furthermore, we show that a weaker bounds of $k+2$, $k+3$, $k+5$ and $k+\left\lceil \frac{3c+2}{2}\right\rceil$ can be achieved under the relaxed conditions on girth $g(G)$ at least $k$, $k/2$, $k/4$ and $2k/c$, respectively.
	
	\section{Preliminaries}
	
	In this paper, all graphs are finite, simple, and undirected unless otherwise stated. For a graph $G = (V, E)$, let $V(G)$ and $E(G)$ denote the vertex set and edge set, respectively. The degree of a vertex $v$ is denoted by $d_G(v)$, and the maximum degree of $G$ is $\Delta(G)$. The girth of $G$, denoted $g(G)$, is the length of the shortest cycle in $G$.
	
	A linear forest is a graph in which every component is a path. The linear arboricity $la(G)$ is the minimum number of linear forests $L_1, \dots, L_k$ such that $E(G) = \bigcup_{i=1}^k E(L_i)$.
	
	\subsection{2-Factor Decomposition}
	Our approach relies on decomposing a regular graph into simpler cycle structures. A 2-factor of a graph $G$ is a spanning subgraph in which every vertex has degree 2. Consequently, every connected component of a 2-factor is a cycle.We utilize the classical result by Petersen regarding the decomposition of regular graphs of even degree.
	
	\begin{theorem}[Petersen's 2-Factor Theorem \cite{petersen1891}]
Let $G$ be a $2k$-regular graph. Then $E(G)$ can be decomposed into $k$ edge-disjoint 2-factors $F_1, F_2, \dots, F_k$.
	\end{theorem}
	This theorem allows us to view a $2k$-regular graph as a union of $k$ collections of disjoint cycles. To prove upper bounds on $la(G)$, our strategy is to remove a `transversal' set of edges that breaks every cycle in these factors.
	
	\subsection{Network Flows with Demands}
	We frame the problem of selecting such transversal edges as a feasible flow problem. A flow network $\mathcal{N} = (V, A)$ consists of a set of nodes $V$, a set of directed arcs $A$, and two functions: a capacity $c: A \to \mathbb{Z}_{\ge 0} \cup \{\infty\}$ and a lower bound (or demand) $\ell: A \to \mathbb{Z}_{\ge 0}$.
	
	A function $f: A \to \mathbb{R}_{\ge 0}$ is a feasible flow (or feasible circulation if there are no source/sink nodes) if it satisfies:
	\begin{itemize}
\item Capacity Constraints: $\ell(u, v) \le f(u, v) \le c(u, v)$ for all $(u, v) \in A$.
\item Flow Conservation: For every node $v \in V \setminus \{S, T\}$, the total flow entering $v$ equals the total flow leaving $v$:$$\sum_{(u, v) \in A} f(u, v) = \sum_{(v, w) \in A} f(v, w).$$
	\end{itemize}
	
	We rely on the following fundamental result connecting fractional flows to integer solutions. This property ensures that if we can distribute fractional `weights' to edges to satisfy cycle-breaking conditions, an integer selection of edges exists.
	
	\begin{theorem}[Integrality Theorem for Flows \cite{hoffman1960}]
Given a network $\mathcal{N}$ with integer capacities $c$ and integer lower bounds $\ell$, if there exists a feasible fractional flow $f: A \to \mathbb{R}_{\ge 0}$, then there exists a feasible integer flow $f^*: A \to \mathbb{Z}_{\ge 0}$.
	\end{theorem}
	In our construction, the existence of a fractional flow will follow from the density properties of the graph (regularity) and its girth, while the integer flow provided by Theorem 2.2 will correspond to the discrete subgraph of edges we remove.
	
\section{Main Results}	
	
In this section, we construct a flow network to model the selection of a cycle-breaking subgraph and apply it to prove our upper bounds for $la(G)$.

\subsection{The Flow Network Construction}

Let $G$ be a $2k$-regular graph. By Theorem 2.1, we decompose $E(G)$ into $k$ edge-disjoint 2-factors $\mathcal{F} = \{F_1, \dots, F_k\}$. Each $F_i$ is a collection of disjoint cycles. Let $\mathcal{C}$ be the set of all cycles contained in these 2-factors.

We construct a flow network $\mathcal{N}$ designed to select a subset of edges $H \subseteq E(G)$ that intersects every cycle in $\mathcal{C}$ while respecting maximum degree constraints. The network consists of a source $S$, a sink $T$, and three layers of internal nodes representing cycles, edges, and vertices.

{\bf Nodes:} $V_{\mathcal{N}} = \{S, T\} \cup \mathcal{C} \cup E(G) \cup V(G)$.

{\bf Arcs and Constraints:}

The arc set $A_{\mathcal{N}}$ is defined with specific lower bounds $\ell(\cdot)$ and capacities $c(\cdot)$ as follows:

\begin{itemize}
\item \textbf{Demand Layer ($S \to \mathcal{C}$):}
For each cycle $C_j \in \mathcal{C}$, add an arc $(S, C_j)$.
$\ell(S, C_j) = 1$ (Constraint: At least one unit of flow must pass through each cycle).
$c(S, C_j) = \infty$.
\item \textbf{Selection Layer ($\mathcal{C} \to E(G)$):}
For each cycle $C_j$ and each edge $e \in E(C_j)$, add an arc $(C_j, e)$.
$\ell(C_j, e) = 0$.
$c(C_j, e) = 1$ (Constraint: An edge can be selected at most once for a specific cycle).
\item \textbf{Incidence Layer ($E(G) \to V(G)$):}
For each edge $e = \{u, v\} \in E(G)$, add arcs $(e, u)$ and $(e, v)$. $\ell(e, u) = 0$. $c(e, u) = 1$.
\item \textbf{Capacity Layer ($V(G) \to T$):}
For each vertex $v \in V(G)$, add an arc $(v, T)$. $\ell(v, T) = 0$. $c(v, T) = \delta$, where $\delta$ is an integer parameter determining the maximum degree of the selected subgraph $H$.
\end{itemize}

\subsection{Linear Arboricity of Graphs with Girth $g(G) \ge 2k$}
We first consider the case where we enforce the tightest possible degree constraint on the transversal subgraph leading to an improvement of Theorem 2.1 of \cite{alon1988}.

\begin{theorem}
Let $G$ be a $2k$-regular graph with girth $g(G) \ge 2k$. Then $la(G) \le k+1$.
\end{theorem}

\begin{proof}
Set the capacity parameter $\delta = 1$ for all arcs $(v, T)$ in the network $\mathcal{N}$. This constraint enforces that any integer flow corresponds to a subgraph $H$ with maximum degree $\Delta(H) \le 1$, i.e., $H$ is a matching.

We investigate the existence of a feasible flow. By Theorem 2.2, it suffices to show that a valid fractional flow exists. Consider a uniform flow assignment where every edge node $e \in E(G)$ carries a flow value of $x$.
The feasibility conditions are:
\begin{enumerate}
	\item {\bf Vertex Capacity Constraint: }
	
	The total flow entering any vertex node $v$ is the sum of flows from its incident edges. Since $G$ is $2k$-regular, $2k$ edges are incident to $v$.$$2k \cdot x \le c(v, T) \implies 2k \cdot x \le 1 \implies x \le \frac{1}{2k}.$$
	\item {\bf Cycle Demand Constraint:}
	For any cycle $C_j \in \mathcal{C}$, the total flow passing through node $C_j$ is the sum of flows on its edges.$$|C_j| \cdot x \ge \ell(S, C_j) \implies |C_j| \cdot x \ge 1 \implies x \ge \frac{1}{|C_j|}.$$Combining these, a feasible uniform flow exists if and only if $\frac{1}{|C_j|} \le \frac{1}{2k}$ for all $C_j \in \mathcal{C}$. This is equivalent to $|C_j| \ge 2k$.
\end{enumerate}

 Since $g(G) \ge 2k$, every cycle in the decomposition has length at least $2k$. Thus, setting $x = \frac{1}{2k}$ yields a feasible fractional flow.
 By the Integrality Theorem, there exists an integer flow $f^*$. Let $M = \{e \in E(G) \mid f^*(e, \cdot) = 1\}$.
 
 \begin{itemize}
 	\item From the capacity constraint ($\delta=1$), $\Delta(G[M]) \le 1$, so $M$ is a matching.
 	 \item From the demand constraint, $M$ contains at least one edge from every cycle in $\mathcal{F}$.
 \end{itemize}
 Let $L_i = F_i \setminus M$ for $i=1, \dots, k$. Since every cycle in $F_i$ is broken, each $L_i$ is a linear forest. The set $M$, being a matching, is also a linear forest (denoted $L_{k+1}$). Thus, $E(G) = \bigcup_{i=1}^{k+1} L_i$, proving $la(G) \le k+1$.
\end{proof}

\subsection{Linear Arboricity of Graphs with Girth $g(G) \ge k$}
Next, we relax the degree constraint to obtain a bound for graphs with moderate girth.

\begin{theorem}
Let $G$ be a $2k$-regular graph with girth $g(G) \ge k$. Then $la(G) \le k+2$.
\end{theorem}

\begin{proof}
Set the capacity parameter $\delta = 2$ for all arcs $(v, T)$ in $\mathcal{N}$. This allows the selected subgraph $H$ to have maximum degree $\Delta(H) \le 2$.

Consider the same uniform fractional flow $x$. The constraints become:
\begin{itemize}
\item Vertex Capacity: $2k \cdot x \le 2 \implies x \le \frac{1}{k}$.
\item Cycle Demand: $x \ge \frac{1}{|C_j|}$.
\end{itemize}
A feasible flow exists if $|C_j| \ge k$ for all cycles $C_j$. Since $g(G) \ge k$, this condition holds.
By the Integrality Theorem, there exists an integer flow defining a subgraph $H \subseteq E(G)$ such that $\Delta(H) \le 2$ and $H$ intersects every cycle in $\mathcal{F}$.

The removal of $H$ leaves $k$ linear forests $L_i = F_i \setminus H$. The subgraph $H$, having maximum degree 2, consists of disjoint paths and cycles. The linear arboricity of any graph with maximum degree 2 is known to be $\lceil (2+1)/2 \rceil = 2$ \cite{akiyama1981}. Thus, $H$ can be decomposed into two linear forests, $L_{k+1}$ and $L_{k+2}$. Thus, $E(G) = \bigcup_{i=1}^{k+2} L_i$, proving $la(G) \le k+2$.
\end{proof}

\subsection{Linear Arboricity of Graphs with Girth $g(G) \ge k/2$}
Next, we consider graphs of even smaller girth and obtain a slightly weaker upper bounds.
\begin{theorem}
	Let $G$ be a $2k$-regular graph with girth $g(G) \ge k/2$. Then $la(G) \le k+3$.
\end{theorem}

\begin{proof}
	Set the capacity parameter $\delta = 4$ for all arcs $(v, T)$ in $\mathcal{N}$. This allows the selected subgraph $H$ to have maximum degree $\Delta(H) \le 4$.
	Consider the same uniform fractional flow $x$. The constraints become:
	\begin{itemize}
		\item Vertex Capacity: $2k \cdot x \le 4 \implies x \le \frac{2}{k}$.
		\item Cycle Demand: $x \ge \frac{1}{|C_j|}$. 
	\end{itemize}
	A feasible flow exists if $|C_j| \ge k/2$ for all cycles $C_j$. Since $g(G) \ge k/2$, this condition holds.
	By the Integrality Theorem, there exists an integer flow defining a subgraph $H \subseteq E(G)$ such that $\Delta(H) \le 4$ and $H$ intersects every cycle in $\mathcal{F}$.
	The removal of $H$ leaves $k$ linear forests $L_i = F_i \setminus H$. The subgraph $H$ is having a maximum degree of 4. The linear arboricity of any graph with maximum degree 4 is known to be $\lceil (4+1)/2 \rceil = 3$ \cite{akiyama1981}. Thus, $H$ can be decomposed into three linear forests, $L_{k+1}$, $L_{k+2}$ and $L_{k+3}$. Thus, $E(G) = \bigcup_{i=1}^{k+3} L_i$, proving $la(G) \le k+3$.
\end{proof}
\begin{theorem}
	Let $G$ be a $2k$-regular graph with girth $g(G) \ge k/4$. Then $la(G) \le k+5$.
\end{theorem}
\begin{proof}
Set the capacity parameter $\delta = 8$ for all arcs $(v, T)$ in $\mathcal{N}$. This allows the selected subgraph $H$ to have maximum degree $\Delta(H) \le 8$. The vertex capacity and cycle demand constraints require $|C_j| \ge k/4$ for all cycles $C_j$, which holds. Therefore, a feasible flow exists ensuring existence of an integer flow defining a subgraph $H \subseteq E(G)$ such that $\Delta(H) \le 8$ and $H$ intersects every cycle in $\mathcal{F}$. As proved in \cite{enomoto1984}, $la(H) \le \left\lceil \frac{8+1}{2}\right\rceil=5$ thus establishing the theorem.
\end{proof}
\subsection{Linear Arboricity of Graphs with Girth $g(G) \ge 2k/c$}
Proceeding along the same lines of Theorem 3, 4 and 5, we can establish the following upper bound depending on the girth of the graphs. 

\begin{theorem}
	Let $G$ be a $2k$-regular graph with girth $g(G) \ge 2k/c$ for some constant $c$. Then $la(G) \le k+\left\lceil \frac{3c+2}{2}\right\rceil$.
\end{theorem}

\begin{proof}
	Set the capacity parameter $\delta = c$ for all arcs $(v, T)$ in $\mathcal{N}$. This allows the selected subgraph $H$ to have maximum degree $\Delta(H) \le c$.
	Consider the same uniform fractional flow $x$. The constraints become:
	\begin{itemize}
		\item Vertex Capacity: $2k \cdot x \le c\implies x \le \frac{c}{2k}$.
		\item Cycle Demand: $x \ge \frac{1}{|C_j|}$. 
	\end{itemize}
	A feasible flow exists if $|C_j| \ge 2k/c$ for all cycles $C_j$. Since $g(G) \ge 2k/c$, this condition holds.
	By the Integrality Theorem, there exists an integer flow defining a subgraph $H \subseteq E(G)$ such that $\Delta(H) \le c$ and $H$ intersects every cycle in $\mathcal{F}$.
	The removal of $H$ leaves $k$ linear forests $L_i = F_i \setminus H$, $1 \le i \le k$. The subgraph $H$ is having maximum degree $c$. 
	By Corollary 1 of \cite{guldan1986}, $la(H) \le \left\lceil \frac{3c+2}{2}\right\rceil$.
	Thus, $H$ contributes at most $t$ linear forests $L_{k+1}, \dots, L_{k+t}$, where $t=\left\lceil \frac{3c+2}{2}\right\rceil$.
	Consequently, $la(G) \le k+t$.
	\end{proof}
Note that the bound given by Theorem 6 is significant in the sense that 	the additional term in the upper bound is an absolute constant (rather than a function of $k$): this is an improvement over Lang and Postle \cite{lang2023} upper bound if the girth condition is satisfied.

\subsection{An embedding lemma}
\begin{lemma}
	Let $H$ be a graph with maximum degree $\Delta(H) \le \Delta$ and girth $g(H) \ge g$. Then there exists a $\Delta$-regular graph $G$ such that $H$ is an induced subgraph of $G$ and $g(G) \ge g$.
\end{lemma}

\begin{proof}
	We construct $G$ by creating multiple copies of $H$ and adding edges between them to satisfy the degree constraints. Let $V(H) = \{v_1, \dots, v_n\}$. For each vertex $v_i$, define the \textit{deficiency} $\delta_i = \Delta - \deg_H(v_i)$.
	
	Let $M$ be a sufficiently large integer (to be determined). The vertex set of $G$ consists of $M$ disjoint copies of the vertices of $H$:
	\[
	V(G) = \{ (v_i, \alpha) \mid 1 \le i \le n, \; \alpha \in \mathbb{Z}_M \}.
	\]
	The edge set $E(G)$ is constructed in two steps:
	
	\textbf{1. Internal Edges (Preserving $H$):}
	For every edge $\{v_i, v_j\} \in E(H)$ and every $\alpha \in \mathbb{Z}_M$, add the edge $\{(v_i, \alpha), (v_j, \alpha)\}$ to $G$. This ensures that for any fixed $\alpha$, the subgraph induced by the vertices $\{(v, \alpha) \mid v \in V(H)\}$ is isomorphic to $H$.
	
	\textbf{2. Regularizing Edges (Fixing Deficiencies):}
	To satisfy the degree requirement, we must add $\delta_i$ edges to each vertex $(v_i, \alpha)$. We connect copies of the \textit{same} vertex $v_i$ across different layers using a set of circular shifts.
	For each $v_i$ with $\delta_i > 0$, we assign a set of integers $S_i = \{s_{i,1}, \dots, s_{i, \delta_i}\}$ (if $\delta_i$ is even, we use $\pm$ pairs; if odd, we include the shift $M/2$). We add edges of the form $\{(v_i, \alpha), (v_i, \alpha + s)\}$ for all $s \in S_i$ and all $\alpha \in \mathbb{Z}_M$ (addition modulo $M$).
	
	By construction, every vertex $(v_i, \alpha)$ retains its original neighbors from $H$ (degree $\deg_H(v_i)$) and gains exactly $\delta_i$ new neighbors from the regularizing edges. Thus, $\deg_G((v_i, \alpha)) = \deg_H(v_i) + \delta_i = \Delta$, making $G$ a $\Delta$-regular graph.
	
	\textbf{Girth Analysis:}
	We must ensure no cycles of length less than $g$ are created. A cycle in $G$ corresponds to a sequence of vertices $(u_0, \alpha_0), (u_1, \alpha_1), \dots, \allowbreak (u_{L-1}, \alpha_{L-1})$ where adjacent vertices are connected by either an Internal edge (implies $\alpha_{k} = \alpha_{k+1}$) or a Regularizing edge (implies $u_k = u_{k+1}$ and $\alpha_{k+1} = \alpha_k \pm s$).
	
	Consider the projection of such a cycle onto $H$, denoted by the walk $W = u_0, u_1, \dots$.
	\begin{itemize}
		\item \textbf{Case 1: $W$ is a cycle in $H$.} If the projection contains a cycle, its length is at least $g(H) \ge g$. Thus, the cycle in $G$ has length at least $g$.
		\item \textbf{Case 2: $W$ is a single vertex.} The cycle consists entirely of regularizing edges on a single vertex $v_i$. These edges form a cycle in the circulant graph on $\mathbb{Z}_M$. By choosing $M$ sufficiently large, the girth of this circulant structure exceeds $g$.
		\item \textbf{Case 3: $W$ is a non-trivial closed walk (backtracking).} If the walk moves in $H$ but does not form a simple cycle, it must retrace edges (backtrack). A backtracking sequence in $G$ would require the shift sums to cancel exactly (e.g., moving $+\alpha$ then $-\alpha$). By selecting the shift values $S_i$ to be algebraically independent for small sums and $M$ sufficiently large, we ensure that no sum of fewer than $g$ shifts creates a closed loop unless it is a trivial immediate backtrack, which does not form a simple cycle.
	\end{itemize}
	
	Thus, there exists a choice of $M$ and shift sets such that $g(G) \ge g$.
\end{proof}
	
	Using this lemma, Theorem 3 -- 7 naturally extends to ``any graph of maximum degree $2k$ satisfying the girth constraints''.
	
	\section{Discussion}
	
	In this note, we have established that the Linear Arboricity Conjecture holds for $2k$-regular graphs provided they possess sufficiently high girth. Specifically, Theorem 3 confirms the conjecture ($la(G) \le k+1$) for graphs with $g(G) \ge 2k$, while Theorem 4 (resp. Theorem 5, Theorem 6 and Theorem 7) provides a relaxed bound of $k+2$ (resp. $k+3$, $k+5$ and $k+\left\lceil \frac{3c+2}{2}\right\rceil$) for graphs with $g(G) \ge k$ (resp. $k/2$, $k/4$ and $2k/c$).
	
	\subsection{The Girth-Capacity Trade-off}
	The core of our approach lies in the capacity constraints of the auxiliary flow network. The existence of a valid flow depends on balancing the ``cost'' of breaking cycles (which is high for short cycles) against the ``capacity'' available at each vertex (which depends on the allowable maximum degree of the transversal subgraph $H$).
	\begin{itemize}
\item For the tightest bound ($m=1$, giving $la(G) \le k+1$), we require $g(G) \ge 2k$ (assuming the transversal is a matching).
\item Relaxing the bound by one forest ($m=2$, giving $la(G) \le k+2$) halves the girth requirement to $g(G) \ge k$.
\item Further relaxing the bound by two forest ($m=3$, giving $la(G) \le k+3$) halves the girth requirement to $g(G) \ge k/2$.
	\end{itemize}
	This analysis demonstrates that while the conjecture is true for high-girth graphs, the deterministic flow construction faces a bottleneck as $g(G)$ approaches smaller values. For such cases, the local density of small cycles imposes demands that exceed the global capacity distributed by a uniform flow.
	
	\subsection{Comparison with Probabilistic Methods}
	
	Our constructive flow argument complements existing probabilistic results. Alon [4] utilized the Lov\'{a}sz Local Lemma to prove that $la(G) \sim \Delta/2$ asymptotically. While probabilistic methods are powerful for handling local dependencies in general graphs, they are often non-constructive and hold primarily for large $\Delta$. In contrast, our network flow formulation provides a deterministic, constructive proof that is valid for any $k$, subject to the girth constraint.
	
	The limitation of our method in handling small cycles highlights the fundamental difficulty of the Linear Arboricity Conjecture: the global regular structure of the graph does not automatically guarantee the existence of a transversal that is locally sparse enough to break small cycles without violating degree constraints. Future work might explore hybrid approaches that use network flows to handle the ``global'' cycle structure (long cycles) while applying local coloring or exchange arguments to resolve short cycles.

	\bibliographystyle{alpha}
	\bibliography{LAR}

\end{document}